\documentclass[12pt]{amsart}

\usepackage{graphicx,amsmath,amscd,amsfonts,amsthm,amssymb,verbatim,stmaryrd,fullpage}
\newtheorem{theorem}{Theorem}
\newtheorem{lemma}[theorem]{Lemma}
\newtheorem{prop}[theorem]{Proposition}

\theoremstyle{remark}

\newcommand{\R}{\mathbb R}

\newcommand{\Z}{\mathbb Z}
\newcommand{\Q}{\mathbb Q}

\newcommand{\LL}{\mathcal L}
\newcommand{\FF}{\mathcal F}
\newcommand{\DD}{\mathcal D}

\newcommand{\ord}{\text{ord}}

\newcommand{\beq}{\begin{equation*}}
\newcommand{\eeq}{\end{equation*}}

\begin{document}

\title{Zero Repulsion in Families of Elliptic Curve $L$-Functions and an Observation of S. J. Miller}

\author{Simon Marshall}
\address{Department of Mathematics\\
Northwestern University\\
2033 Sheridan Road\\
Evanston\\
IL 60202, USA}
\email{slm@math.northwestern.edu}

\begin{abstract}
We provide a theoretical explanation for an observation of S. J. Miller that if $L(s,E)$ is an elliptic curve $L$-function for which $L(1/2, E) \neq 0$, then the lowest lying zero of $L(s,E)$ exhibits a repulsion from the critical point which is not explained by the standard Katz-Sarnak heuristics.  We establish a similar result in the case of first-order vanishing.
\end{abstract}

\maketitle

\section{Introduction}

In the paper \cite{Mi}, S. J. Miller investigated the statistics of the zeros of various families of elliptic curve $L$-functions.  His key observation is illustrated in Figure \ref{fig:miller} below.  It shows a histogram of the first zero above the central point for rank zero elliptic curve $L$-functions generated by randomly selecting the coefficients $c_1$ up to $c_6$ in the defining equation

\begin{equation*}
y^2 + c_1 xy + c_3 y = x^3 + c_2 x^2 + c_4 x + c_6
\end{equation*}
for curves with conductors in the ranges indicated in the caption.  The zeros are scaled by the mean density of low zeros of the $L$-functions and the plots are normalized so that they represent the probability density function for the first zero of $L$-functions from this family.  Miller observes that there is clear repulsion of the first zero from the central point; that is, the plots drop to zero at the origin, indicating a very low probability of finding an $L$-function with a low first zero.  This is quite surprising, as one would expect to be able to model the distribution of the lowest zeros by the eigenvalue distribution of matrices in $SO(2N)$, with $N$ chosen to be equal to half the logarithm of the conductor of the curve.  However, it is well known in random matrix theory that matrices in $SO(2N)$ exhibit no repulsion of the first eigenvalue from the origin.\\

\begin{figure}
\includegraphics[width=6.48cm]{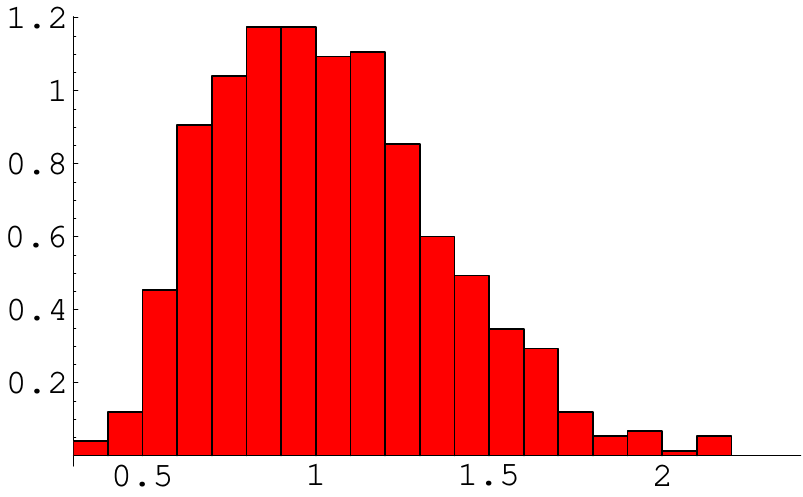} \includegraphics[width=6.48cm]{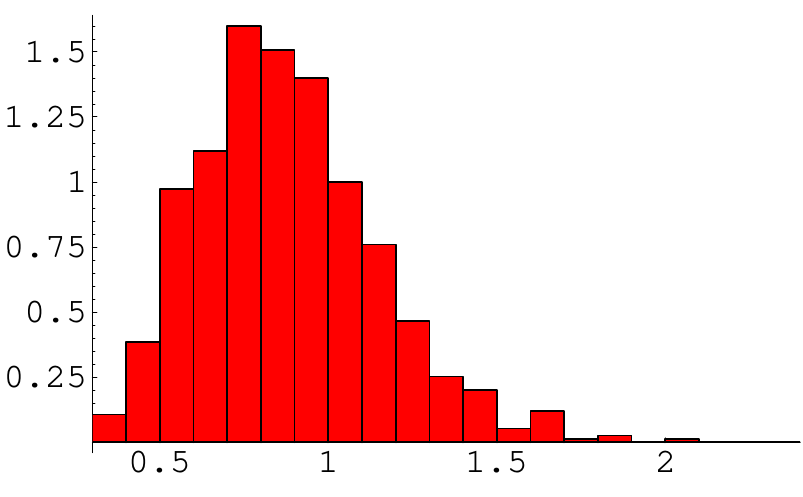}
\caption{\label{fig:miller} First normalized zero above the
central point: Left: $750$ rank 0 curves from $y^2 +
a_1xy+a_3y=x^3+a_2x^2+a_4x+a_6$, $\log({\rm cond}) \in [3.2,
12.6]$. Right: $750$ rank 0 curves from
$y^2 + a_1xy+a_3y=x^3+a_2x^2+a_4x+a_6$, $\log({\rm cond}) \in
[12.6, 14.9]$.}
\end{figure}

In this paper, we prove a result which provides theoretical support for Miller's observation.  Rather than considering algebraic families of elliptic curves, we consider quadratic twists of a fixed curve $E/\Q$.  Let $d \neq 1$ be a fundamental discriminant which is relatively prime to the conductor $M$ of $E$, and let $\chi_d$ be the associated quadratic character of modulus $|d|$.  We shall consider the families $\FF^+(E)$ and $\FF^-(E)$ of all even and odd twists of $E$ by the characters $\chi_d$, and define $\DD^+(E)$ and $\DD^-(E)$ to be the corresponding sets of fundamental discriminants.  In addition, we define the subfamilies $\FF^0(E) \subset \FF^+(E)$ and $\FF^1(E) \subset \FF^-(E)$ by

\begin{equation*}
\FF^i(E) = \{ L(s, E \times \chi_d) \: | \: \ord_{s = 1/2} L(s, E \times \chi_d) = i \}, \quad i = 0, 1,
\end{equation*}
together with the associated sets of discriminants

\begin{equation*}
\DD^i(E) = \{ d \: | \: \ord_{s = 1/2} L(s, E \times \chi_d) = i \}.
\end{equation*}
As is standard, we shall assume that all $L$-functions under consideration satisfy the Riemann hypothesis, and for any $d$ we define $\gamma_d$ to be the height of the lowest non-real zero of $L(s, E \times \chi_d)$.  Our main result is the following.

\begin{theorem}
\label{repel}
For all $d \in \DD^0(E)$, we have

\begin{equation}
\label{twistbound}
\frac{ \ln |\gamma_d| }{ \ln |d| } \ge -1/4 + O_{E,\epsilon}( ( \ln \ln |d| )^{-1+\epsilon}).
\end{equation}
In the odd case, there exists a nonzero integer $d_0$ satisfying $(d_0, 2M) = 1$ such that if we define

\begin{equation*}
\DD^-_{d_0}(E) = \{ d \: | \: (d, 2 M d_0) = 1, d_0 d < 0, d_0 d \text{ is square mod } 4M \},
\eeq
then $\DD^-_{d_0}(E) \subseteq \DD^-(E)$ and for all $d \in \DD^1(E) \cap \DD^-_{d_0}(E)$ we have

\begin{equation*}
\frac{ \ln |\gamma_d| }{ \ln |d| } \ge -1/4 + O_{E,\epsilon}( ( \ln \ln |d| )^{-1+\epsilon}).
\end{equation*}

\end{theorem}

\subsection{Relation to the Katz-Sarnak heuristics}

We shall first discuss the relationship between Theorem \ref{repel} and the Katz-Sarnak heuristics in the simpler case of even functional equation.  While the zero repulsion of Theorem \ref{repel} is only on a scale of $d^{-1/4 + o(1)}$, which is very small in comparison with the mean zero spacing of $\sim \ln |d|^{-1}$, the fact that it holds for all members of the (conjecturally) large family $\FF^0(E)$ means that one can infer properties of the distribution of rescaled lowest zeros from it under some natural assumptions.  Let $D > 0$ be given, and let $\LL_D$ be the multiset of rescaled lowest zeros,

\begin{equation*}
\LL_D = \{ \gamma_d \ln |d| \: | \: d \in \DD^0(E), D/2 \le |d| \le D \}.
\end{equation*}
Suppose that the set $\LL_D$ has a limiting distribution of the form $\rho(x) dx$ as $D \rightarrow \infty$, where $\rho(x)$ is a smooth density on $[0, \infty)$.  If $\rho(x)$ vanishes to order $r$ at the origin and we make the natural assumption that $| \LL_D | \gg D$ as $D \rightarrow \infty$, then we have

\begin{equation*}
P( \exists \, x \in \LL_D | x \le D^{-1/(r+1)} ) \ge \delta > 0
\end{equation*}
for $D$ sufficiently large.  Therefore, if $\LL_D$ has a limiting distribution of the type described for which $r \le 2$, we obtain a contradiction to Theorem \ref{repel}.  This is evidence that any limiting distribution of $\LL_D$ has to vanish to order at least three at the origin, in agreement with Miller's data.  Moreover, if one believes the `minimalist conjecture' that $\FF^0(E)$ makes up a density one subset of $\FF^+(E)$, then our result shows that the lowest zeros of the family $\FF^+(E)$ do not obey the standard Katz-Sarnak heuristics.  For reasons which will become apparent in the course of the proof of Theorem \ref{repel}, we believe that any such discrepancy should be viewed as a consequence of the special value formula

\begin{equation}
\label{wald0}
L(1/2, E \times \chi_d) = \kappa_E \frac{ c_E(|d|)^2}{ |d|^{1/2}}, \quad c_E(|d|) \in \Z,
\end{equation}
of Waldspurger \cite{W}, Kohnen-Zagier \cite{KZ} and Baruch-Mao \cite{BM}.\\

The tension between Miller's data and the standard $SO(2N)$ model for even twists of a fixed elliptic curve has also been considered by Due\~nez, Huynh, Keating, Miller and Snaith in \cite{HKS}.  They propose a modification to the $SO(2N)$ model which they term an \textit{excised orthogonal ensemble}, and which exhibits the observed repulsion from the critical point.  Their excised ensemble is the subset of $SO(2N)$ consisting of matrices $g$ whose characteristic polynomial $P(g,x)$ is not small at 1, which reflects the inequality

\begin{equation}
\label{Waldspurger}
L(1/2, E \times \chi_d ) \gg_E |d|^{-1/2}, \quad d \in \DD^0(E)
\end{equation}
that follows from the special value formula (\ref{wald0}).  We note that this gap is also the basis of our proof of Theorem \ref{repel} in the even case.  The natural choice of cut-off in the excised ensemble of \cite{HKS} is $| P(g, 1)| \ge C|d|^{-1/2}$ for some constant $C$, and with this choice they prove that their probability distribution, denoted $R_1^{T_{\mathcal X}}(\theta) d\theta$, exhibits a hard gap around the origin similar to that of Theorem \ref{repel}.  More precisely, they show that $R_1^{T_{\mathcal X}}(\theta) = 0$ for $|\theta| \ll d^{-1/4 + \delta}$ (Theorem 1.3, \cite{HKS}), where $\delta > 0$ depends on fitted parameters and may be made arbitrarily small.  Moreover, it seems likely that they would have vanishing of $R_1^{T_{\mathcal X}}(\theta)$ for $|\theta| \ll |d|^{-1/4 + o(1)}$ under some equidistribution assumption on the zeros of the matrices in their ensemble.  This is analogous to the way that obtaining an exponent of $-1/4$ in Theorem \ref{repel} requires controlling the bulk of the zeros of $L(s, E \times \chi_d)$.\\

In the case of odd functional equation, we now define $\LL_D$ to be

\begin{equation*}
\LL_D = \{ \gamma_d \ln |d| \: | \: d \in \DD^1(E) \cap \DD^-_{d_0}(E), D/2 \le |d| \le D \}.
\end{equation*}
The minimalist conjecture would again imply that $\DD^1(E) \cap \DD^-_{d_0}(E)$ is a density one subsequence of $\DD^-_{d_0}(E)$, and because this has $\gg D$ elements of size at most $D$, we may again combine the odd case of Theorem \ref{repel} with the above argument to deduce that any smooth limiting distribution of $\LL_D$ must vanish to order at least three.  Applying the minimalist conjecture once more, we obtain the same result for the lowest zeros in the family $\FF^-_{d_0}$ corresponding to $\DD^-_{d_0}$.  The Katz-Sarnak heuristics predict that these zeros should have the same distribution as the first nontrivial eigenvalue of a random matrix in $SO(2N+1)$, but this distribution only vanishes to second order at the origin (see pages 10-11 and 411-416 of \cite{KS1}, or page 10 of \cite{KS2}).  Similarly to the even case, we shall see that this discrepancy may be viewed as a consequence of the formula of Gross-Zagier \cite{GZ} for the central derivative of $L(s, E \times \chi_d)$, together with a lower bound for the height of a nontorsion point on an elliptic curve as in Anderson-Masser \cite{AM}.

{\bf Acknowledgements}:  We would like to thank Peter Sarnak, Eduardo Due\~nez, Duc Khiem Huynh, Jonathan Keating, Steven J. Miller, Nina Snaith and David Hansen for many helpful discussions, and Nicolas Templier for explaining the correct application of the Gross-Zagier formula in the odd case to us.

\section{Proof of Theorem \ref{repel} in the even case}

We shall first present the proof of Theorem \ref{repel} in the case of even functional equation, which combines the special value inequality (\ref{Waldspurger}) with a complex analytic argument.  To extend the proof to the odd case, we only need to replace this inequality with an analogue for the central derivative, which may be derived from the Gross-Zagier formula and a height gap for points on elliptic curves.  This is carried out in section \ref{oddsect}.

It should be noted that if one is content to replace the error term in Theorem \ref{repel} with $o(1)$ then the result follows immediately from the fact that the completed $L$-function $\Lambda(s,E \times \chi_d)$ has vanishing central derivative, and (after suitable normalisation) its central value satisfies (\ref{Waldspurger}) while its second derivative on the critical line is $\ll |d|^\epsilon$ by Lindel\"of.  Our refinement proceeds by combining (\ref{Waldspurger}) with Jensen's formula in the ball of radius 1 at the central point.  Let $\{ \gamma_{d,n} | n \ge 1 \}$ be the multiset of imaginary parts of zeros of $L(s, E \times \chi_d )$ (that is, counted with multiplicity), and define $Q = \ln \ln |d|$.  Note that all implied constants will be assumed to depend on $E$ from now on.  Jensen's formula states that

\begin{eqnarray}
\notag
 \sum_{ 1 \ge | \gamma_{d,n} | } \ln | \gamma_{d,n} | + \frac{1}{2\pi} \int_0^{2\pi} \ln | L( 1/2 + e^{i\theta}, E \times \chi_d ) | d\theta & = & \ln | L(1/2, E \times \chi_d ) | \\
\notag
\sum_{ Q^{-1+\epsilon} \ge | \gamma_{d,n} | } \ln | \gamma_{d,n} | + \sum_{ 1 \ge | \gamma_{d,n} | \ge Q^{-1 + \epsilon} } \ln | \gamma_{d,n} | \qquad \qquad & \quad & \\
\label{jensen}
 + \frac{1}{2\pi} \int_0^{2\pi} \ln | L( 1/2 + e^{i\theta}, E \times \chi_d ) | d\theta & \ge & -1/2 \ln |d| + O(1),
\end{eqnarray}
where we have broken up the sum on the LHS of (\ref{jensen}) with the aid of an arbitrarily chosen parameter $\epsilon$ satisfying $1 > \epsilon > 0$.  We shall prove the following asymptotics for the last two terms on the LHS of (\ref{jensen}):

\begin{prop}
\label{jzerosprop}
For every $1 > \epsilon > 0$, the second term in (\ref{jensen}) satisfies the asymptotic

\begin{equation}
\label{jzeros}
\sum_{ 1 \ge |\gamma_{d,n}| \ge Q^{-1+\epsilon} } \ln |\gamma_{d,n}| = -\frac{2}{\pi} \ln |d| + O_\epsilon( \ln |d| Q^{-1+\epsilon} ).
\end{equation}

\end{prop}

\begin{prop}
\label{jbdryprop}

For every $1 > \epsilon > 0$, the third term in (\ref{jensen}) satisfies the bound

\begin{equation}
\label{jbdry}
\frac{1}{2\pi} \int_0^{2\pi} \ln | L( 1/2 + e^{i\theta}, E \times \chi_d ) | d\theta \le \frac{2}{\pi} \ln |d| + O_\epsilon( \ln |d| Q^{-1+\epsilon} ).
\end{equation}

\end{prop}

After substituting the asymptotics of Propositions \ref{jzerosprop} and \ref{jbdryprop} into equation (\ref{jensen}), we have

\begin{equation*}
\sum_{ Q^{-1+\epsilon} \ge | \gamma_{d,n} | } \ln | \gamma_{d,n} | \ge -1/2 \ln |d| + O_\epsilon( \ln |d| Q^{-1+\epsilon} ) + O(1),
\end{equation*}
and Theorem \ref{repel} follows from this on throwing away all but the lowest pair of conjugate zeros in the sum.

\subsection{Estimating sums over zeros}

Proposition \ref{jzerosprop} will follow by integration by parts from the following estimate for the counting function $N_d(a,b)$, defined by

\begin{equation*}
N_d(a,b) = | \{ \gamma_{d,n} | a \le \gamma_{d,n} \le b \} |.
\end{equation*}

\begin{lemma}
\label{countest}

We have the asymptotic

\begin{equation}
N_d(a,b) = \frac{ (b-a) \ln |d| }{\pi} + O_\epsilon( \ln |d| Q^{-1+\epsilon} )
\end{equation}
uniformly for $a$ and $b$ in any compact interval.

\end{lemma}
Lemma \ref{countest} will be proven in turn by applying the explicit formula of the $L$-function $L(s, E \times \chi_d)$.  The particular form that we shall use is the one stated in Proposition 2.1 of \cite{RS}.  To recall it, let $g \in C^\infty_0(\R)$ be given, and define $h(r) = \int_{-\infty}^\infty g(x) e^{irx} dx$.  Let the conductor of $L(s,E)$ be $M$.  If we have

\begin{equation*}
L(s,E) = \sum_{n=1}^\infty \frac{a(n)}{n^s},
\end{equation*}
we define $c_d(n) = \Lambda(n) \chi_d(n) a(n)$ so that

\begin{equation*}
-\frac{L'}{L} (s, E \times \chi_d) = \sum_{n=1}^\infty \frac{c_d(n)}{n^s}.
\end{equation*}
With these notations, the explicit formula may be stated as follows:

\begin{prop}[{\bf The explicit formula}]
The zeros $\{ \gamma_{d,n} \}$ and coefficients $c_d(n)$ satisfy the relation

\begin{multline}
\label{explicit}
\sum_{\gamma_{d,n} } h( \gamma_{d,n} ) = \frac{1}{2\pi} \int_{-\infty}^\infty h(r) \left( \ln (M d^2 ) - 2 \ln 2\pi + \frac{ \Gamma' }{ \Gamma } ( 1 + ir) + \frac{ \Gamma' }{ \Gamma } ( 1 - ir) \right) dr \\
- \sum_{n=1}^\infty \left( \frac{ c_d(n) }{ \sqrt{n} } g( \ln n) + \frac{ \overline{c_d(n)} }{ \sqrt{n} } g( -\ln n) \right).
\end{multline}

\end{prop}
We shall first use the explicit formula to prove a smooth version of Lemma \ref{countest}.  To state it, let $f_0 \in C^\infty_0(\R)$ be a real even function with support in $[ -1/2, 1/2]$, and define $f = f_0 * f_0$.  We then define

\begin{equation*}
f_Q(x) = f(x/Q), \qquad \widehat{f}(r) = \int_{-\infty}^\infty f(x) e^{irx} dx, \qquad \widehat{f_Q}(r) = \int_{-\infty}^\infty f_Q(x) e^{irx} dx. \\
\end{equation*}
Our construction of $f$ implies that $\widehat{f}(r) \ge 0$, and we may normalise $f$ so that $\int \widehat{f}(r) dr = 1$.  Let $I(r)$ be the characteristic function of the interval $[a,b]$, and define $I_Q$ to be the smooth approximation $I * \widehat{f_Q}$ to $I$.  We then have

\begin{lemma}
\label{smoothest}

We have the asymptotic 

\begin{equation*}
\sum_{ \gamma_{d,n} } I_Q( \gamma_{d,n} ) = \frac{(b-a) \ln |d|}{\pi} + O(\ln |d|^{2/3} ),
\end{equation*}
uniformly for $a$ and $b$ in any compact interval.

\end{lemma}

\begin{proof}

Define $g_Q$ to be the function

\begin{eqnarray*}
g_Q(x) & = & f_Q(x) \frac{-1}{2\pi} \int_a^b e^{-i tx} dt \\
& = & \frac{1}{2 \pi ix} (e^{-ibx} - e^{-iax}) f_Q(x),
\end{eqnarray*}
so that $I_Q = \widehat{g_Q}$.  Substituting the pair $g_Q, I_Q$ into the explicit formula yields

\begin{multline}
\label{explicit2}
\sum_{ \gamma_{d,n} } I_Q( \gamma_{d,n} ) = \frac{1}{2\pi} \int_{-\infty}^\infty I_Q(r) \left( \ln (M d^2 ) - 2 \ln 2\pi + \frac{ \Gamma' }{ \Gamma } ( 1 + ir) + \frac{ \Gamma' }{ \Gamma } ( 1 - ir) \right) dr \\
- \sum_{n=1}^\infty \left( \frac{ c_d(n) }{ \sqrt{n} } g_Q( \ln n) + \frac{ \overline{c_d(n)} }{ \sqrt{n} } g_Q( -\ln n) \right).
\end{multline}
Because $I_Q(r)$ has rapid decay in $r$ at a rate which is uniform as $Q \rightarrow \infty$, and $\int_{-\infty}^\infty I_Q(r) dr = (b-a)$, we may simplify the integral in (\ref{explicit2}) so that it becomes

\begin{equation}
\label{explicit3}
\sum_{ \gamma_{d,n} } I_Q( \gamma_{d,n} ) = \frac{(b-a) \ln |d|}{\pi} + O(1) - \sum_{n=1}^\infty \left( \frac{ c_d(n) }{ \sqrt{n} } g_Q( \ln n) + \frac{ \overline{c_d(n)} }{ \sqrt{n} } g_Q( -\ln n) \right).
\end{equation}
We may bound the sum in (\ref{explicit3}) using the following properties of $g_Q$:

\begin{eqnarray*}
\text{supp} (g_Q) & \subseteq & [ -Q, Q], \\
|g_Q(x)| & \ll & |f_Q(x)| \left| \int_a^b e^{-i tx} dt \right| \\
& \ll & |a-b| \ll 1.
\end{eqnarray*}
With these, we have

\begin{eqnarray}
\notag
\left| \sum_{n=1}^\infty \frac{ c_d(n) }{ \sqrt{n} } g_Q( \ln n) \right| & \le & \sum_{n=1}^{e^Q } \frac{ |c_d(n)| }{ \sqrt{n} } \\
\notag
& \ll_\epsilon & \sum_{n=1}^{e^Q } n^{-1/2+\epsilon} \\
\notag
& \ll_\epsilon & e^{(1/2+\epsilon)Q} \\
\label{sumbound}
& = & \ln |d|^{(1/2 + \epsilon)},
\end{eqnarray}
and likewise for the second term.  Note that we have used the fact that the coefficients $a(n)$ satisfy the Ramanujan bound $|a(n)| \ll n^\epsilon$, which implies the same bound for $|c_d(n)|$.  Substituting (\ref{sumbound}) into (\ref{explicit3}) and relaxing $1/2 + \epsilon$ to $2/3$ gives the result.

\end{proof}

We now deduce Lemma \ref{countest} from Lemma \ref{smoothest} by comparing $I_Q$ with $I$.  If $r$ lies outside the intervals $[a - Q^{-1+\epsilon}, a + Q^{-1+\epsilon}]$ and $[b - Q^{-1+\epsilon}, b + Q^{-1+\epsilon}]$, the difference $|I(r) - I_Q(r)|$ may be estimated as

\begin{eqnarray}
\notag
|I(r) - I_Q(r)| & \le & \left| \left( \int_{-\infty}^{-Q^{-1+\epsilon}} + \int_{Q^{-1+\epsilon}}^\infty \right) \widehat{f_Q}(r) dr \right| \\
\notag
& = & \left| \left( \int_{-\infty}^{-Q^{-1+\epsilon}} + \int_{Q^{-1+\epsilon}}^\infty \right) Q \widehat{f}(Qr) dr \right| \\
\notag
& = & \left| \left( \int_{-\infty}^{-Q^{\epsilon}} + \int_{Q^{\epsilon}}^\infty \right) \widehat{f}(r) dr \right| \\
\label{IIQ}
& \ll_{\epsilon,A} & Q^{-A}
\end{eqnarray}
for any $A$, by the rapid decay of $\widehat{f}$.  The positivity of $I_Q$ means that we may estimate the sum $\sum I_Q( \gamma_{d,n} )$ from below by its restriction to $\gamma_{d,n} \in [a + Q^{-1+\epsilon}, b - Q^{-1+\epsilon}]$, after which we may apply (\ref{IIQ}) and Lemma \ref{smoothest} to bound $N_d(a,b)$ from above as follows:

\begin{eqnarray*}
\sum_{\gamma_{d,n} } I_Q( \gamma_{d,n} ) & \ge & \sum_{b - Q^{-1+\epsilon} \ge \gamma_{d,n} \ge a + Q^{-1+\epsilon}} I_Q( \gamma_{d,n} ) \\
\frac{(b-a) \ln |d|}{\pi} + O( \ln |d|^{2/3}) & \ge & (1 + O_\epsilon(Q^{-2})) N_d(a + Q^{-1+\epsilon}, b - Q^{-1+\epsilon} ).
\end{eqnarray*}
Replacing $a$ and $b$ with $a - Q^{-1+\epsilon}$ and $b + Q^{-1+\epsilon}$ yields the upper bound in Lemma \ref{countest}.

We derive a lower bound for $N_d(a, b)$ in a similar way, starting from the inequality

\begin{eqnarray*}
(1 + O_\epsilon(Q^{-2}) ) N_d(a - Q^{-1+\epsilon}, b + Q^{-1+\epsilon} ) & \ge & \sum_{b + Q^{-1+\epsilon} \ge \gamma_{d,n} \ge a - Q^{-1+\epsilon}}  I_Q( \gamma_{d,n} ) \\
& = & \frac{(b-a) \ln |d|}{\pi} + O(\ln |d|^{2/3} ) \\
& \quad & \quad - \sum_{ \gamma_{d,n} \notin (a - Q^{-1+\epsilon}, b + Q^{-1+\epsilon} ) }  I_Q( \gamma_{d,n} )
\end{eqnarray*}
The bounds

\begin{equation*}
| I_Q(r)| \ll_{\epsilon,A} (Qr)^{-A}, \quad r \notin (a - Q^{-1+\epsilon}, b + Q^{-1+\epsilon} )
\end{equation*}
and $N_d(-T, T) \ll \ln |d| T \ln T$ for $T \ge 1$ (see for instance Theorem 5.8, V, p. 104 \cite{IK}) imply that

\begin{equation*}
\sum_{ \gamma_{d,n} \notin (a - Q^{-1+\epsilon}, b + Q^{-1+\epsilon} ) }  |I_Q( \gamma_{d,n} )| \ll_\epsilon \ln |d| Q^{-2},
\end{equation*}
and so

\begin{equation*}
N_d(a - Q^{-1+\epsilon}, b + Q^{-1+\epsilon} ) \ge \frac{(b-a) \ln |d|}{\pi} + O( \ln |d| Q^{-2}).
\end{equation*}
The lower bound in Lemma \ref{countest} again follows from this by substitution.

It remains to derive Proposition \ref{jzerosprop} from Lemma \ref{countest}.  We first apply integration by parts, which gives

\begin{eqnarray*}
\sum_{ 1 \ge \gamma_{d,n} \ge Q^{-1+\epsilon} } \ln \gamma_{d,n} & = & N_d(0,x) \ln x \bigg|_{Q^{-1+\epsilon}}^1 - \int_{Q^{-1+\epsilon}}^1 N_d(0,x) \frac{dx}{x} \\
& = & - N_d(0, Q^{-1+\epsilon}) \ln Q^{-1+\epsilon} - \int_{Q^{-1+\epsilon}}^1 N_d(0,x) \frac{dx}{x}.
\end{eqnarray*}
After substituting the asymptotic of Lemma \ref{countest}, this becomes

\begin{eqnarray*}
\sum_{ 1 \ge \gamma_{d,n} \ge Q^{-1+\epsilon} } \ln \gamma_{d,n} & = & O_\epsilon ( \ln |d| Q^{-1+\epsilon} \ln Q) - \int_{Q^{-1+\epsilon}}^1 \left( \frac{x}{\pi} \ln |d| + O_\epsilon ( \ln |d| Q^{-1+\epsilon} ) \right) \frac{dx}{x} \\
& = & -(1-Q^{-1+\epsilon}) \frac{ \ln |d| }{\pi} + O_\epsilon( \ln |d| Q^{-1+\epsilon} \ln Q) \\
& = & -\frac{ \ln |d| }{\pi} + O_\epsilon( \ln |d| Q^{-1+\epsilon}),
\end{eqnarray*}
and Proposition \ref{jzerosprop} follows from this by the symmetry of $\{ \gamma_{d,n} \}$.

\subsection{Proof of Proposition \ref{jbdryprop}}

We shall prove Proposition \ref{jbdryprop} with the aid of the following explicit subconvex bound for $L(s, E \times \chi_d)$ (Theorem 5.19, V, p. 116 \cite{IK}), which is valid for any $L$-function satisfying Riemann and Ramanujan.

\begin{prop}
\label{subconvex}
We have the bound

\begin{equation}
\ln |L(s, E \times \chi_d )| \ll \frac{ ( \ln |d| )^{2 - 2\sigma} }{ (2\sigma - 1) Q } + Q
\end{equation}
for all $s$ satisfying $|s - 1/2| \le 1$.

\end{prop}

We divide the integral in Proposition \ref{jbdryprop} into three regions, the first containing the points within $Q^{-1+\epsilon}$ of the critical line, and the second and third the semicircles to the right and left of the critical line.  In the first region, the length of the arc of integration and the convex estimate $\ln |L(s, E \times \chi_d )| \ll \ln |d|$ imply that the integral is $\ll \ln |d| Q^{-1+\epsilon}$.  In the right semicircle, we apply Proposition \ref{subconvex} and the bound $\sigma - 1/2 \ge Q^{-1+\epsilon}$ to obtain

\begin{eqnarray}
\notag
\ln |L(s, E \times \chi_d )| & \ll & \frac{ (\ln |d|)^{1 - 2Q^{-1+\epsilon} } }{ Q^{-1+\epsilon} Q } + Q \\
\notag
& = & \ln |d| \exp( - 2Q^{-1+\epsilon} \ln \ln |d|) Q^{-\epsilon} + Q \\
\notag
& = & \ln |d| \exp( -2 Q^\epsilon ) Q^{-\epsilon} + Q \\
\label{rightarc}
& \ll_\epsilon & \ln|d| Q^{-1}.
\end{eqnarray}
This implies that the integral over the right semicircle is also $O_\epsilon( \ln |d| Q^{-1})$.  In the left semicircle, we combine (\ref{rightarc}) with the functional equation to see that

\begin{equation}
\ln | L(s, E \times \chi_d ) | \le (1-2\sigma) \ln |d| + O_\epsilon( \ln |d| Q^{-1} ).
\end{equation}
This implies that the integral over the left semicircle is $\le 2\ln |d| / \pi + O_\epsilon( \ln |d| Q^{-1} )$, which completes the proof.

\section{The proof in the odd case}
\label{oddsect}

We begin the proof of Theorem \ref{repel} in the case of odd functional equation by applying the results of Bump-Friedberg-Hoffstein \cite{BFH}, Murty-Murty \cite{MM}, and Waldspurger \cite{W1, W2} to deduce the existence of a quadratic character $\chi_{d_0}$ with $(2M, d_0) = 1$ such that $L(1/2, E \times \chi_{d_0}) \neq 0$.  This is the $d_0$ that we shall take in the statement of Theorem \ref{repel}.  If we choose a second fundamental discriminant $d$ satisfying

\beq
(d, 2Md_0) = 1, \quad d_0 d <0, \quad d_0 d \equiv \Box \: (4M),
\eeq
then the discussion on pages 268-269 of \cite{GZ} shows that $L(s, E \times \chi_d)$ has odd functional equation and hence $\DD^-_{d_0}(E) \subseteq \DD^-(E)$ as claimed.  We may deduce Theorem \ref{repel} from the following lemma in exactly the same manner as in the case of even functional equation, by applying Jensen's formula to $L(s, E \times \chi_d) / (s - 1/2)$.

\begin{lemma}

For $d \in \DD^1(E) \cap \DD^-_{d_0}(E)$, we have

\beq
L'(1/2, E \times \chi_d) \gg |d|^{-1/2}.
\eeq

\end{lemma}

\begin{proof}

The conditions that $(d, d_0) = 1$ and $d_0 d < 0$ allow us to define a genus class character $\chi_{d, d_0}$ of the imaginary quadratic field $K = \Q( \sqrt{d_0 d})$ (see for instance \cite{IK}, chapter 22, page 510).  We base-change $E$ to $K$ and twist by $\chi_{d, d_0}$, and denote the associated $L$-function by $L(s, E_K \times \chi_{d, d_0})$.  As the discriminant of $K$ is prime to $2M$ and all primes dividing $M$ split in $K$, we may apply the Gross-Zagier formula \cite{GZ} to the central derivative $L'(1/2, E_K \times \chi_{d, d_0})$ to obtain

\begin{equation}
\label{gross}
L'(1/2, E_K \times \chi_{d, d_0}) = C_E h(P_{d, d_0}) |d_0d|^{-1/2}.
\end{equation}
Here, $P_{d, d_0}$ is an algebraic point on $E$ which is defined over $\Q(\sqrt{d})$ (see the discussion on pages 268-269 of \cite{GZ}), and $h( \cdot )$ denotes the canonical height on $E$.  The Kronecker factorisation formula and the fact that $L(1/2, E \times \chi_d) = 0$ imply that

\begin{eqnarray*}
L(s, E_K \times \chi_{d, d_0}) & = & L(s, E \times \chi_{d_0}) L(s, E \times \chi_d) \\
L'(1/2, E_K \times \chi_{d, d_0}) & = & L(1/2, E \times \chi_{d_0}) L'(1/2, E \times \chi_d).
\end{eqnarray*}
Substituting this into (\ref{gross}) and using our assumption that $L(1/2, E \times \chi_{d_0}) \neq 0$, we have

\beq
L'(1/2, E \times \chi_d) = C_{E, d_0} h(P_{d, d_0}) |d|^{-1/2}.
\eeq
We now apply the results of Anderson-Masser \cite{AM}, which state that if $P$ is a nontorsion algebraic point on $E$ whose degree over $\Q$ is bounded, then the height $h(P)$ must be bounded from below.  Our assumption that $L'(1/2, E \times \chi_d) \neq 0$ implies that $P_{d, d_0}$ is not torsion, from which we then deduce that

\beq
L'(1/2, E \times \chi_d) \gg |d|^{-1/2}
\eeq
as required.

\end{proof}

\end{document}